\documentclass[10pt,a4paper]{amsart}

\usepackage{setspace}
\usepackage[utf8x]{inputenc}
\usepackage[T1]{fontenc}

\usepackage{helvet}

\usepackage[geometry]{ifsym}
\usepackage{amsthm}
\usepackage{amsmath}
\usepackage{verbatim}
\usepackage{amssymb} 
\usepackage{resizegather}
\usepackage{tikz}
\usepackage{tikz-cd}
\usepackage{babel} 
\usepackage{mathrsfs}
\usepackage{graphicx}
\usepackage{epstopdf}
\usepackage{epigraph}

\usepackage{pgfplots}
\usepgfplotslibrary{fillbetween}
\pgfplotsset{compat=1.13}

\usepackage[bottom]{footmisc}

\usepackage{hyperref}

\newtheorem{theorem}{Theorem}[section]

\newtheorem{proposition}{Proposition}[theorem]

\newtheorem{lemma}{Lemma}[theorem]

\newtheorem{corollary}{Corollary}[theorem]

\newtheorem{remark}{Remark}
\newtheorem{example}{Example}
\newtheorem{definition}{Definition}
\newtheorem{hypothesis}{Hypothesis}[section]

\title{On the Gauss-Manin Connection and Real Singularities}
\author{Lars Andersen}
\begin{document}

\maketitle
\begin{abstract}
    We prove that to each real singularity $f: (\mathbb{R}^{n+1}, 0) \to (\mathbb{R}, 0)$ one can associate two systems of differential equations $\mathfrak{g}^{k\pm}_f$ which are pushforwards in the category of $\mathcal{D}$-modules over $\mathbb{R}^{\pm}$, of the sheaf of real analytic functions on the total space of the positive, respectively negative, Milnor fibration. We prove that for $k=0$ if $f$ is an isolated singularity then $\mathfrak{g}^{\pm}$ determines the the $n$-th homology groups of the positive, respectively negative, Milnor fibre. We then calculate $\mathfrak{g}^{+}$ for ordinary quadratic singularities and prove that under certain conditions on the choice of morsification, one recovers the top homology groups of the Milnor fibers of any isolated singularity $f$. As an application we construct a public-key encryption scheme based on morsification of singularities.
\end{abstract}
\begin{center}
    \text{Classification: }\textbf{14-XX, 94-XX}
\end{center}

\section{Introduction}

Ever since 1958 when Y. Manin (\cite{Man58}) proved that the homology groups of a Milnor fibre of an isolated complex analytic singularity $f: (\mathbb{C}^{n+1}, 0)\to (\mathbb{C}, 0)$ can be computed by solving certain systems of differential equations, called the Gauss-Manin system, the theory of such systems, or connections, has played a huge role in the theory of linear differential equations (by way of $\mathcal{D}$-module theory, which has arisen as a substantial subject  of research since the 1970's) as well as in singularity theory. It gives a way of not only computing the integral homology of the Milnor fibres $H_n(\mathcal{F}_{\eta, \mathbb{C}})$ but also find the monodromy group of the singularity.\\

However no such theory seems to be at our disposal when we deal with real singularities\footnote{A notable exception is an article \cite{barlet} from 2002 by D. Barlet which connects in a beautiful way $\mathcal{D}$-module theory over the complex numbers with the study of real singularities. There might be other such exceptions but not to our knowledge.}. The aim of this article is to begin to amend that, and show that the corresponding real Gauss-Manin differential equations are not only completely different for those of a complexification, but that they can be used to compute the top homology group $H_n(\mathcal{F}_{\eta}^{\pm})$.  
\subsection{A Real Milnor Fibration}
Suppose given a real analytic map germ $f: (\mathbb{R}^{n+1}, 0)\to (\mathbb{R}, 0)$ such that any representative of $(V(f), 0)$ has an isolated singular point in the origin. Then (cf. \cite{Milnor}) there exists $\delta_0>0$ such that for any $\delta\in(0,\delta_0]$ there exists $\epsilon_0>0$ such that for any $\epsilon\in (0, \epsilon]$ the maps 
$$f: \mathcal{N}_{\epsilon(\delta)}^{+}=f^{-1}((0, \epsilon])\cap \mathbb{B}_{\delta}\to (0,\epsilon]\subset \mathbb{R}^{+},$$
$$f: \mathcal{N}_{\epsilon(\delta)}^{-}=f^{-1}([-\epsilon,0))\cap \mathbb{B}_{\delta}\to (0,\epsilon]\subset \mathbb{R}^{-},$$
are trivial $C^{\infty}$-fibrations called the positive, respectively negative open Milnor fibrations of $f$. The fibers (over $\eta$) are denoted $\mathcal{F}_{\eta}^{+}$ respectively $\mathcal{F}_{\eta}^{-}$ and are called the positive and negative Milnor fibres. By the Regular Value Theorem (see e.g. \cite{Kosinski}) they are $C^{\infty}$-manifolds.
\subsection{Notation}
For any $N\in\mathbb{N}$ and for any real analytic submanifold $M\subset \mathbb{R}^{N}$ let $\mathcal{O}_{M}$ denote the sheaf of real analytic functions on $M$.\\
In particular if $(U,\mathcal{O}_U)$ is a complex analytic space with $U\supset M$ an open neighborhood of $M$ then $\mathcal{O}_M=\mathcal{O}_{U|M}$.\\ 
We will only work with real analytic manifolds but for a lucid exposition on real analytic geometry and especially coherent sheaves we refer to the book \cite{italian}.\\
Given a real analytic space $(X, \mathcal{O}_X)$ we say the $X$ is Stein if for any coherent $O_X$-module $\mathcal{C}$, 
$$H^k(X, \mathcal{C})=0,\qquad\forall k>0$$
where the homology is coherent sheaf cohomology.\\
Given real analytic spaces $(X,\mathcal{O}_X)$ and $(Y, \mathcal{O}_Y)$ and a morphism $f: X\to Y$ we will say that $f$ is \emph{Stein} if given a real analytic subspace $Z\subset Y$ which is Stein, then the pullback of $(Z, \mathcal{O}_Z)$ by $f$ is Stein.\\
By the word $\mathcal{D}$-module we will mean a module over the sheaf of differential operators with coefficients real analytic functions. 
Our main reference for $\mathcal{D}$-module theory is the excellent work \cite{Pham} and we follow its terminology for $\mathcal{D}$-modules. For instance $\int_f^{\ast}$ denotes higher direct images in the category of $\mathcal{D}$-modules and $DR_X(\mathcal{O}_X)$ denotes the de Rham complex
$$\Omega_X^n\leftarrow^d \Omega_X^{n-1}\leftarrow^d\dots \leftarrow^d \Omega_X^0=\Omega_X$$
with $d$ the standard exterior derivative on differential $k$-forms. 
\section{$\mathcal{D}$-modules over the real numbers}
\subsection{The Gauss-Manin Connection}
The following holds.

\begin{theorem}\label{main theorem} There exists a ring isomorphism 
$$\int_{f}^k \mathcal{O}_{\mathcal{N}^{+}_{\epsilon(\delta)}}\cong H^{k+n}(\mathcal{F}_{\eta}^{+}; \mathbb{R})\otimes_{\mathbb{R}} \mathcal{O}_{\mathbb{R}^{+}, \eta}$$
\end{theorem}
\begin{proof} 
\begin{enumerate}
\item We claim that $f: N_{\epsilon(\delta)}^{+}\to \mathbb{R}^{+}$ is Stein. First note that the vector field $\text{grad}(f)$ is real analytic. Therefore the integral flow $\phi_t$ by $\text{grad}(f)$ of the lift of the vector field $\partial/\partial t$ on $\mathbb{R}^+$ is real analytic as can be seen by using Picard iteration, as in the proof of the Picard-Lindelöf theorem on the local solutions of systems of ODE's. It follows that there exists a real analytic isomorphism
$$h: f^{-1}(I_{\epsilon})\cap \mathbb{B}_{\delta}\to \left(f^{-1}(\eta)\cap \mathbb{B}_{\delta}\right)\times I_{\epsilon}.$$
Hence
$$H^{\cdot}(\mathcal{N}_{\epsilon(\delta)}^{+}; f^{\ast}\mathcal{G})\cong H^{\cdot}(\mathcal{F}_{\eta}^{+}\times I_{\epsilon}; h_{\ast}f^{\ast}\mathcal{G})$$
$$\cong H^{\cdot}(\mathcal{F}_{\eta}^{+}\times I_{\epsilon}; \pi_1^{\ast} f^{\ast}\mathcal{G}_{|\mathcal{F}_{\eta}^{+}}\otimes \pi_2^{\ast}\mathcal{G})\cong \bigoplus_{i+j=\cdot} H^{i}(\mathcal{F}_{\eta}^{+}, f^{\ast}\mathcal{G}_{|\mathcal{F}_{\eta}^{+}})\otimes H^{j}(I_{\epsilon}, \mathcal{G}).$$
Since $\mathcal{G}$ is coherent and $I_{\epsilon}$ is a real analytic manifold (hence coherent hence a real analytic space), $H^{j}(I_{\epsilon}, \mathcal{G})=0$ for $j>0$, by \cite[Theorem III.3.7]{italian}. Furthermore $f^{\ast}\mathcal{G}$ is coherent since all the spaces involved are locally Noetherian,\\
so $H^{i}(\mathcal{F}_{\eta}^{+}, f^{\ast}\mathcal{G}_{|\mathcal{F}_{\eta}^{+}})=0$ for $i>0$, again by the \cite[Theorem III.3.7]{italian}. This proves the claim.
\item 
One the one hand one has by definition of $\int_f^{\cdot}$ that
$$\int_f^k \mathcal{O}_{\mathcal{N}_{\epsilon(\delta)}^{+}}=\mathbb{R}f_{\ast}^k(\mathcal{D}_{\mathcal{N}^{+}_{\epsilon(\delta)}\leftarrow \mathbb{R}^+}\otimes_{\mathcal{D}_{\mathcal{N}_{\epsilon(\delta)}^{+}}}^L \mathcal{O}_{\mathcal{N}_{\epsilon(\delta)}^+})$$
and on the other hand one has that $DR(\mathcal{D}_{\mathcal{N}_{\epsilon(\delta)}^+})$ is by the purely algebraic result \cite[Lemma 4.3.5]{Pham} a locally free resolution (hence a projective resolution) of the transition module $\mathcal{D}_{\mathcal{N}_{\epsilon(\delta)}^+\leftarrow \mathbb{R}^+}$. And since $f$ is Stein we can apply \cite[Proposition 14.3.4]{Pham} to deduce that the RHS is quasi-isomorphic to
$$\mathbb{R}^{k+n}f_{\ast}DR_{\mathcal{N}_{\epsilon(\delta)}^+/\mathbb{R}^+}(\mathcal{O})=\mathbb{R}^{k+n} f_{\ast}(\Omega_{\mathcal{N}_{\epsilon(\delta)}^+/\mathbb{R^+}}^{\cdot}).$$
Using the relative Poincaré Lemma (\cite{Kulikov}[section 3.3]) then gives 
$$\int_f^k \mathcal{O}_{\mathcal{N}_{\epsilon(\delta)}^+}=H^{k+n}(\mathcal{N}_{\epsilon(\delta)}^+/\mathbb{R}^+; \mathbb{R})$$
in the notation of \cite{Pham}. Since $f_{|\mathcal{N}_{\epsilon(\delta)}^+}$ is a trivial $C^{\infty}$-fibration this sheaf is constant of fiber $H^{k+n}(f_{|\mathcal{N}_{\epsilon(\delta)}^+}^{-1}(\eta); \mathbb{R})$ over $\eta\in \mathbb{R}^+$. This finishes the proof.
\end{enumerate}
\end{proof}
\begin{remark} The theorem only gives information about $H^n(\mathcal{F}_{\eta}; \mathbb{R})$ but this might not say very much. Indeed it is \emph{not} true that $H_n(\mathcal{F}_{\eta})$ is necessarily non-zero. An $ADE$-singularity $\mathcal{F}_{\eta}$ will in general have the homology of a $k$-sphere with $0\leq k\leq n$, by the results in a previous article \cite{Lars}.
\end{remark}
\begin{remark}
From now onwards, we shall for brevity sometimes omit the subindices and write simply $\mathcal{N}^{\pm}$ for the total spaces of the real Milnor fibrations.
\end{remark}
We now make the following
\begin{definition}
Let $k\in\mathbb{N}$. The $k$-th positive, respectively negative, Gauss-Manin system associated to $f$ is the $\mathcal{O}_{\mathbb{R}^+}$-module, respectively $\mathcal{O}_{\mathbb{R}^-}$-module,
$$\int_f^k \mathcal{O}_{\mathcal{N}^+},\qquad\int_f^k \mathcal{O}_{\mathcal{N}^-}.$$
\end{definition}

Let us recall from \cite{Pham} that the systems $\int_f^k \mathcal{O}_{\mathcal{N}^+}$, and $\int_f^k \mathcal{O}_{\mathcal{N}^-}$ are naturally endowed with a $\mathcal{D}$-module structure. In analogy with \cite{Pham} we now analyse further these $\mathcal{D}$-modules. Define the module of multiple coats\footnote{In french "module de couches multiples" has another, quite humorous meaning} relative to $f: \mathcal{N}^{+}\to \mathbb{R}^+$ by the exact sequence 
$$\mathcal{O}_{\mathcal{N}^+\times \mathbb{R}^+}\to \mathcal{O}_{\mathcal{N}^+\times \mathbb{R}^+}[\frac{1}{t-f}]\to \mathfrak{B}_f\to 0$$
One defines the relative de Rham complex of $\mathfrak{B}_f$ from the sheaves of relative differentials by the exact sequence
\begin{equation}\label{short}
    0\to \Omega_{\mathcal{N}^+\times \mathbb{R}^+, \mathbb{R}^+}\to \Omega_{\mathcal{N}^+\times \mathbb{R}^+, \mathbb{R}^+}[\frac{1}{t-f}]\to DR_{\mathcal{N}^+\times\mathbb{R}^+, \mathbb{R}^+}(\mathfrak{B}_f) \to 0
\end{equation}

In the following proposition we let $\pi: \mathcal{N}^{+}\times\mathbb{R}^{+}\to \mathbb{R}^{+}$ denote the standard projection onto the second factor.
\subsection*{Notation} Any equality in the proof means isomorphism is the appropriate derived category.

\begin{proposition}\label{prop one} There is an isomorphism 
$$\int_f^k \mathcal{O}_{\mathcal{N}^+}\cong H^{n+k}(\pi_{\ast}DR_{\mathcal{N}^+\times \mathbb{R}^+, \mathbb{R}^+}(\mathfrak{B}_{f, 0})),\qquad k=-n,\dots, 0$$ 
\end{proposition}
\begin{proof}
Write for simplicity $X$ and $Y$ instead of $\mathcal{N}^+$ and $\mathbb{R}$. Decompose $f$ as the composition $f=\pi \circ i$ where
$$i: X\to X\times Y,\qquad x\mapsto (x, f(x)),$$
$$\pi: X\times Y\to Y,\qquad (x,y)\mapsto y.$$
Then by definition, $\mathcal{B}_{f, 0}=\int_i \mathcal{O}_X$ so that by Proposition 14.3.1 in \cite{Pham}
$$\int_f^k \mathcal{O}_X=\int_{\pi}^k \mathcal{B}_{f, 0}.$$
On the other hand by Proposition 2.4.8 in \cite{Bjork},
$$\int_{\pi}^k \mathcal{B}_{f, 0}=\mathcal{H}^k \pi_{\ast}(DR_{X\times Y| Y}(\mathcal{B}_{f, 0}))[n]$$
which gives the result.
\end{proof}

\subsection{Real Analytic Solutions}
Proceeding as in the holomorphic setting, we construct solutions to $\int_f^0 \mathcal{O}_{\mathcal{N}_{\epsilon(\delta)}^{+}}$ as follows. Let $U\subset\mathbb{R}^+$ be open and define
$$(\int_f^k \mathcal{O}_{\mathcal{N}_{\epsilon(\delta)}^+})(U) \to \mathcal{O}_{\mathbb{R}^+}(U)$$
$$c\mapsto \int_{h(\eta)} c$$
where $h(\eta)\in H^n(\mathcal{F}_{\eta}^+; \mathbb{R})$.
\begin{corollary}\label{corr one}
The application $h\mapsto \int_h$ is an isomorphism
$$H^n(\mathcal{F}_{\eta}^+;\mathbb{R})\to Hom_{\mathcal{D}}(\int_f^0 \mathcal{O}_{\mathcal{N}^{+}_{\epsilon(\delta)}}, \mathcal{O}_{\mathbb{R}^+})$$
\end{corollary}
\begin{proof}
Use de Rham's theorem.
\end{proof}

In other words, one can identify the highest degree cohomology group of the positive Milnor fiber with the real analytic solutions of the positive Gauss-Manin system associated to the singularity.
\begin{remark}
In the holomorphic case one can replace the word 'solution' above with 'horisontal solution'. It is not clear for us whether or not this might be done in the real case.
\end{remark}

\section{Examples}
In this section we will look at an instance where the homology of the Milnor fibres is known. All the examples will therefore be $ADE$-singularities and the goal is to find the associated Gauss-Manin differential equations. 
\begin{example}
Let $f=x^k$ for an integer $k>1$. Then $H_0(\mathcal{F}_{\eta}^{+})=\mathbb{R}\oplus\mathbb{R}$ if $k\in 2\mathbb{Z}$ and $H_0(\mathcal{F}_{\eta}^{+})=\mathbb{R}$ if $k\in\mathbb{Z}\setminus 2\mathbb{Z}$. In either case the homology group is generated by the class $\gamma(\eta)=[f^{-1}(I_{\eta})]$ where $I_{\epsilon}=(0,\epsilon]$, and 
$$u(\eta)=\int_{\gamma(\eta)}\frac{1}{t-x^k}$$
For any $q\in\mathbb{Q}$, and $l\in\mathbb{N}$ let $(q)_l$ denotes its $l$-th Pochammer symbol. For any $a,b, c\in\mathbb{Q}$ let $\mathfrak{F}(a,b;c,\cdot): \mathbb{D}\subset\mathbb{C}\to \mathbb{C}$ be the Gaußian \emph{hypergeometric function}
$$\mathfrak{F}(a, b; c,z)=\sum_{l\geq 0} \frac{(a)_l (b)_l}{(c)_l}\frac{z^n}{l\!}.$$
Then 
$$u(\eta)=\left[\frac{x\mathfrak{F}(1,1/k;1+1/k, x^k/t)}{t}\right]_{\gamma(\eta)}.$$
For any $a,b,c\in\mathbb{Q}$, Gauß proved that the function $\mathfrak{F}$ satisfies the \emph{hypergeometric equation}:
$$x(x-1) D_{xx} u+(c-(a+b+1) x) D_x u-ab u=0.$$
Therefore, if in our case we assume $x\neq 0$, then $\tilde{u}=t u/ x$ satisfies 
$$\frac{x^k}{t}(\frac{x^k}{t}-1) D_{xx} \tilde{u}+ (1+1/k-(2+1/k)\frac{x^k}{t}) D_x\tilde{u}-\frac{1}{k}\tilde{u}=0$$
by the superposition principle for ODE's. Since $D_x\tilde{u}=-t D_x u/x^2$ and $D_{xx}\tilde{u}=2t D_{xx}/x^3$ we get that $u$ satisfies
\begin{equation}\label{one}
2x^{k-3}(\frac{x^k}{t}-1)D_{xx}u-(\frac{k+1}{k}\frac{t}{x^2}-\frac{2k+1}{k} x^{k-2})D_x u-\frac{t}{kx} u=0
\end{equation}
Let $S$ be the differential operator above and let $\mathfrak{g}$ be the Gauss-Manin system of $f$, that is, $\mathfrak{g}=\int_f^0 \mathcal{O}_{\mathcal{N}^+}$. Then it follows that 
 \[\mathfrak{g}= \left\{
    \begin{array}{ll}
      \mathcal{D}/S\mathcal{D},\qquad k\equiv 0\quad(\text{mod 2})\\
      \mathcal{D}/S\mathcal{D}\oplus \mathcal{D}/S\mathcal{D},\qquad  k\equiv 1 \quad(\text{mod 2})
\end{array} \right. \]
\end{example}

We will see later on that the Gauss-Manin system in this example is "odd" in comparison to that of any higher-dimensional ordinary quadratic singularity of the same Morse index. Yet this might not be surprising after all, since $f=x^k$ is unique (in said class of singularities) in having that the top-dimensional homology is principal.

\begin{example} Let $f=x_1^1+x_2^2+x_3^2$. Then $H_2(\mathcal{F}_{\eta}^+)=\mathbb{R}$ is generated by the class $\gamma({\eta})$ of $S_{\sqrt{\eta}}^2$. We use spherical coordinates 
\[
    \begin{bmatrix}
           x_{1} \\
           x_{2} \\
           x_3
         \end{bmatrix}
  =
    \begin{bmatrix}
           \cos{\phi_1} \\
           \sin{\phi_1}\cos{\phi_2}\\
           \sin{\phi_1}\sin{\phi_2}
         \end{bmatrix}
\]
with $\phi_1\in(0,\pi]$ and $\phi_2\in (0,2\pi]$, and calculate
$$\int_0^{2\pi}\int_0^{\pi} \frac{\sin{\phi_1}}{t-\cos^2{\phi_1}\sin^2{\phi_2}-\sin^2{\phi_1}\sin^2{\phi_2}-\cos^2{\phi_1}}=$$
$$2\int_0^{\pi}([\frac{\sin{\phi_1}\arctan{\frac{\tan{\phi_2}(3-2t+\cos(2\phi_1)}{\sqrt{A_t}}}}{\sqrt{A_t}}]_{\phi_2=0}^{2\pi})d\phi_1$$
where we have put 
$$A_t=8t-4t^2-7/2+4(t-1)\cos{2\phi_1}-1/2\cos{4\phi_1}$$
Now, since $\tan{0}=\tan{2\pi}=0$ the entity inside the brackets is zero, hence the integral equals $u(\eta)=2\pi$. As a consequence the Gauss-Manin system is $D_t=0$.
\end{example}

Note that in  the previous examples one would get entirely different Gauss-Manin systems if one were to work over $\mathbb{C}$. The degrees of the ODE's are different in the second example; over $\mathbb{C}$ one would get the system
$$(t D_t-(3/2))u=0,$$
with solutions $u_{\eta}(t)=c t^{3/2}$ for any $\eta\in\mathbb{C}$ such that $0<|\eta|<<t$.

\section{The Case of Ordinary Quadratic Singularities}
According to the results of the article \cite{Lars} the following holds: if $f: (\mathbb{R}^{n+1},0) \to (\mathbb{R},0)$ is an $ADE$-singularity then there exists a non-negative integer $c_d$ such that 
\begin{enumerate}
\item $c_d\leq n$, and 
\item $H_k(\mathcal{F}_{\eta}^+)=0$ $\forall k\neq 0, c_d$,
\item $rank H_{c_d}(\mathcal{F}_{\eta}^+)=1$
    \end{enumerate}
    
However $c_d=n$ if and only if $f=\sum_{i=1}^{n+1} x_i^2$ and so our main theorem is only applicable for $ADE$-singularities if $f$ is of this form. And since there is no list of the Poincaré polynomials of other singularities than such, we are at present fairly limited in constructing examples. To analyse this problem further we will begin with a technical lemma.

\begin{lemma}\label{technical} Let $n\in\mathbb{N}$ be non-zero. If $f=x_1^2+\dots+x_{n+1}^2$ then
$$\int_{\gamma(\eta)}\frac{d\mathbf{x}}{t-f}=2\pi\frac{V_{n-1}(\eta)}{(t-\eta)},$$
where $V_{n-1}=\frac{\pi^{(n-1)/2}}{\Gamma((n-1)/2+1)} \eta^{(n-1)/2}$
\end{lemma}
\begin{proof}
We can without loss of generality assume that $\eta=1$. Since $\gamma(\eta)$ is the class of the unit $n$-sphere $\mathbb{S}^n\subset\mathbb{R}^{n+1}$ we can use spherical coordinates $\phi_1,\dots, \phi_n$ such that $dV=\sin^{n-1}{\phi_1}\dots \sin{\phi_{n-1}}d\phi_1\dots d\phi_n$ and such that
\[
    \begin{bmatrix}
           x_{1} \\
           x_{2} \\
           \vdots \\
           x_n\\
           x_{n+1}
         \end{bmatrix}
  =
    \begin{bmatrix}
           \cos{\phi_1} \\
           \sin{\phi_1}\cos{\phi_2}\\
           \vdots \\
           \sin{\phi_1},\dots,\sin{\phi_{n-1}},\cos{\phi_{n}}\\
           \sin{\phi_1},\dots,\sin{\phi_{n-1}},\sin{\phi_{n}}
         \end{bmatrix}
\]
where $\phi_1,\dots,\phi_{n-1}\in [0,\pi )$ and $\phi_n\in[0, 2\pi)$. A standard inductive argument gives
$$cos^2{\phi_1}+\dots+\sin^2{\phi_1}\dots\cos^2{\phi_n}+\sin^2{\phi_1}\dots\sin^2{\phi_1}=1$$
hence 
$$\int_{\gamma{\eta}}\frac{d\mathbf{x}}{t-f}=\int_{0}^{2\pi}\int_0^{\pi}\dots\int_0^{\pi}\frac{\sin^{n-1}{\phi_1}\dots\sin{\phi_{n-1}}}{t-1} d\mathbf{\phi}=$$
$$\frac{1}{t-1} \int_{\mathbb{S}^n} d\mathbf{x}=\frac{\text{Vol}(\mathbb{S}^n)}{t-1}$$
which gives the result.
\end{proof}
As a consequence we get

\begin{proposition}\label{prop two} Let $n\in\mathbb{N}$ be non-zero. The Gauss-Manin system associated to an  ordinary quadratic singularity $f: (\mathbb{R}^{n+1}, 0)\to (\mathbb{R}, 0)$ of Morse index zero is $\mathcal{D}/{D_t \mathcal{D}}$ with solutions 
$$u_n(\eta)=2\pi\frac{V_{n-1}(\eta)}{(t-\eta)}.$$
\end{proposition}
\begin{proof}
Since $\beta(\mathcal{F}_{\eta}^+)=1+u^n$ by \cite[Corollary 3.0.1]{Lars} we can apply the Theorem \hyperref[main theorem]{\ref*{main theorem}}. In particular  
$$\int_{\cdot} \frac{d\mathbf{x}}{t-f}: H_n(\mathcal{F}_{\eta})\to \text{Hom}_{\mathcal{D}}(\int_f^n \mathcal{O}_{\mathcal{N}^+}, \mathcal{O}_{\mathbb{R}^+})$$ 
is an isomorphism. We then use Lemma \hyperref[technical]{\ref*{technical}} to find its image. This gives the solution $u_n$ to the Gauss-Manin system, which we find by differentiation. Moreover by the Theorem \hyperref[main theorem]{\ref*{main theorem}} this determines the system uniquely as a $\mathcal{D}$-module, because $\text{rank } H_n(\mathcal{F}_{\eta}^+)=1$ by \cite{Lars}. 
\end{proof}

The following theorem is a generalisation of the previous proposition. Here we shall find the image of the map for any ordinary quadratic singularity. So for any $k\in \mathbb{N}$ let $\mathfrak{g}^k$ be a $\mathcal{O}_{\mathbb{R}^+}$-module (in the sense of $\mathcal{D}$-modules) such that $\text{Hom}_{\mathcal{D}}(\mathfrak{g}^k, \mathcal{O}_{\mathbb{R}^{+}})$ is isomorphic to the image $\text{Im}(\int_{\cdot} 1/(t-f))$ (see Theorem \hyperref[main theorem]{\ref*{main theorem}}). 

\begin{theorem}\label{minor theorem} Let $n\in \mathbb{N}$ be non-zero. If  $f: (\mathbb{R}^{n+1}, 0)\to (\mathbb{R}, 0)$ is an ordinary quadratic singularity of Morse index $\lambda$ then $\mathfrak{g}_{n-\lambda}=\mathcal{D}/ D_t \mathcal{D}_t$ and
$\text{Hom}_{\mathcal{D}}(\mathfrak{g}^{n-\lambda}, \mathcal{O}_{\mathbb{R}^+})$ is spanned by 
$$u_n(t,\eta)=2\pi \frac{V_{n-\lambda-1}(\eta)}{t-\eta}$$
\end{theorem}
\begin{proof} Write 
$f=x_1^2+\dots x_{n-\lambda}-x_{n+1-\lambda}-\dots-x_{n+1}$. Then $\beta(\mathcal{F}_{\eta}^+)=1+u^{n-\lambda}$ by \cite[Corollary 3.0.1]{Lars} and we can take $\gamma(\eta)=[f^{-1}(\eta)\cap \{x_{n+1-\lambda}=,\dots, x_{n+1}=0]$ as a generator of $H_{n-\lambda}(\mathcal{F}_{\eta}^+)$. By definition the solutions of $\mathfrak{g}^{n-\lambda}$ are then given by 
$$\int_{\gamma(\eta)} \frac{d\mathbf{x}}{1-f}=\int_{\mathbb{S}_{n-\lambda}}\frac{1}{t-(x_1^2+\dots+x_{n-\lambda}^2)}$$
which by the Lemma \hyperref[technical]{\ref*{technical}} equals $2\pi \frac{V_{n-\lambda-1}}{t-\eta}$. 
\end{proof}

We have not yet been able to generalise Theorem \hyperref[minor theorem]{\ref*{minor theorem}} to other $ADE$-singularities, due to the fact that the integrals becomes rather unruly. And we have yet to analyse the singularities of the systems of differential equations appearing in this section.

\section{Top Homology of the Milnor Fibres}
\subsection{Preliminaries}
For any $s\in\mathbb{R}^n$ set
$$f_s: \mathbb{R}^n\to\mathbb{R}, \quad x\mapsto f(x)+\sum_{i=1}^n s x_i^2$$
$$F: \mathbb{R}^{n}\times\mathbb{R}\to\mathbb{R}, \quad (x,s)\mapsto f_s(x).$$
$$\tilde{F}: \mathbb{R}^{n}\times \mathbb{R}\to \mathbb{R}\times\mathbb{R},\quad (x,s)\mapsto (f_s(x),t)$$
and for any $s_0\in\mathbb{R}^{+}\setminus\{0\}$ set 
$$\mathcal{N}^{+}=\tilde{F}^{-1}((0,\eta]\times[0, s_0])\cap (\mathbb{B}_{\delta}\times \mathbb{R})$$
$$\mathcal{N}_s^{+}=f_s^{-1}((0,\eta])\cap\mathbb{B}_{\delta}.$$
In what follows we will abuse notation and write for brevity $\langle s, x^2 \rangle $ instead of $\sum_{i=1}^n s x_i^2.$
\subsection{The Main Result}
We will assume that that there exists a choice of $s_0$ as above such that the inclusion $i_s: \mathcal{N}^{+}_s\to \mathcal{N}$ is a homotopy equivalence. Furthermore, assume that $f$ together with its partial derivatives are analytic in the origin. We will prove that this implies that the following holds
\begin{theorem}\label{main main theorem}
Suppose the morsification parameter space is of dimension one. If there exists $s_0\in\mathbb{R}^{+}$ such that for any $s\in\mathbb{R}$ if $0\leq s\leq s_0$ then the inclusions $i_s: \mathcal{N}_s^{+}\to \mathcal{N}^+$ are homotopic then there is a quasi-isomorphism

\begin{equation}
    DR^{n-1}_{\mathcal{N}_s^{+}\times\mathbb{R}^+,\mathbb{R}^+}(\mathfrak{B}_{f_s})\to DR^{n-1}_{\mathcal{N}_0^{+}\times \mathbb{R}^+,\mathbb{R}^+}(\mathfrak{B}_f).
\end{equation}
\end{theorem}
\subsection{Consequences of the Theorem}
Consider the cohomology sheaves $$\tilde{\mathcal{H}}_s^{\bullet}:=\mathcal{H}^{\bullet}(\pi_{\ast}DR_{\mathcal{N}_s^{+}\times\mathbb{R}^{+},\mathbb{R}^{+}}(\mathfrak{B}_{f_s})).$$
Since by the results in \cite{Lars} the positive Milnor fiber of $f_s$ at the critical point $p_i\in\text{Crit}(f_s)$ has no cohomology in degree $n-1$ unless $\lambda(p_i)=0$ in which case $$H_{n-1}(f_s^{-1}(s_i+\eta_i)\cap\mathbb{B}_{\delta_i})=\mathbb{Z},\quad H_{n-1}(f_s^{-1}(s_i-\eta_i)\cap\mathbb{B}_{\delta_i})=\{0\},$$
the sheaves $\tilde{\mathcal{H}}_s^{n-1}$ are supported on $\prod_{i| \lambda(p_i)=0} (s_i, s_i+\eta_i]$ where $(\eta_i, \delta_i)$ are local Milnor data for $f_s$ at the critical points $p_i$ corresponding to the critical values $s_i$, for $i=1,\dots, m$. Moreover $$\tilde{\mathcal{H}}_{s,s_i+\eta_i}^{n-1}=\mathcal{H}^{n-1}(DR_{\mathcal{N}_s^{+}\times\mathbb{R}^{+},\mathbb{R}^{+}}(\mathfrak{B}_{f_s}))(f_s^{-1}(s_i+\eta_i)\cap \mathbb{B}_{\delta_i}),$$
$$\tilde{\mathcal{H}}_{0,\eta}^{n-1}=\mathcal{H}^{n-1}(DR_{\mathcal{N}^{+}\times\mathbb{R}^{+},\mathbb{R}^{+}}(\mathfrak{B}_{f}))(f^{-1}(\eta)\cap \mathbb{B}_{\delta}).$$
Now, if one sets
$$\mathcal{H}_s^{\bullet}:=\mathcal{H}^{\bullet}(DR_{\mathcal{N}_s^{+}\times\mathbb{R}^+, \mathbb{R}^+}(\mathfrak{B}_{f_s}))$$
then by Theorem \hyperref[main main theorem]{\ref*{main main theorem}} there is an isomorphism of sheaves $\mathcal{H}_s^{n-1}\cong \mathcal{H}_0^{n-1}$. But then 
$$\mathcal{H}^{n-1}(DR_{\mathcal{N}_s^{+}\times\mathbb{R}^+, \mathbb{R}^+}(\mathfrak{B}_{f_s}))(\bigcup_{i=1}^m f_s^{-1}((s_i,s_i+\eta_i])\cap\mathbb{B}_{\delta_i})\cong$$
$$\mathcal{H}^{n-1}(DR_{\mathcal{N}_0^{+}\times\mathbb{R}^+, \mathbb{R}^+}(\mathfrak{B}_{f}))(f^{-1}((0,\eta])\cap\mathbb{B}_{\delta})$$
by the above. Here we have used the fact that 
$$\mathcal{H}^{n-1}(DR_{\mathcal{N}_s^{+}\times\mathbb{R}^+,\mathbb{R}^+}(\mathfrak{B}_{f_s})(\mathbb{B}_{\delta}\setminus\bigcup_i \mathbb{B}_{\delta_i})=\int_{x} \mathcal{O}_{\mathcal{N}_s^+}$$
is trivial because the Gauss-Manin system of a non-singular function is trivial by construction.
This gives
$$\bigoplus_{i| \lambda(p_i)=0}\mathcal{H}^{n-1}(DR_{\mathcal{N}_s^{+}\times\mathbb{R}^+, \mathbb{R}^+}(\mathfrak{B}_{f_s}))(f_s^{-1}(\eta_i)\cap\mathbb{B}_{\delta_i})\cong$$
$$\mathcal{H}^{n-1}(DR_{\mathcal{N}_0^{+}\times\mathbb{R}^+, \mathbb{R}^+}(\mathfrak{B}_{f}))(f^{-1}(\eta)\cap\mathbb{B}_{\delta})$$
hence
$$\bigoplus_{i| \lambda(p_i)=0} \tilde{\mathcal{H}}^{n-1}_{s, \eta_i}\cong \tilde{\mathcal{H}}_{0,\eta}^{n-1}$$
By Proposition \hyperref[prop one]{\ref*{prop one}} and Theorem \hyperref[main theorem]{\ref*{main theorem}} the right hand side is $H_{n-1}(\mathcal{F}^{+})=\int_f^0 \mathcal{O}_{\mathcal{N}_{\epsilon(\delta)}^+}$ whereas each of the factors on the left hand side is the Gauss-Manin system $\int_{f_s} \mathcal{O}_{\mathcal{N}_{\epsilon_i(\delta_i)}^+}$ of the germ of ordinary quadratic singularity $(f_s, p_i)$ with Milnor data $(\eta_i, \delta_i)$. Thus by Proposition \hyperref[prop one]{\ref*{prop one}} together with Theorem \hyperref[minor theorem]{\ref*{minor theorem}} and Theorem \hyperref[main main theorem]{\ref*{main main theorem}},

$$H_{n-1}(\mathcal{F}^+)\cong \bigoplus_{i} \int_{f_s}^0 \mathcal{O}_{\mathcal{N}_{\epsilon_i(\delta_i)}^{+}}$$
$$=\bigoplus_{p\in Crit(f_s)| \lambda(p)=0}\mathcal{D}/D_t$$

Furthermore, by Theorem \hyperref[minor theorem]{\ref*{minor theorem}} the solution space is generated by a copy of $u_n=2\pi \frac{V_{n-1}}{t-\eta}$ for each point $p\in\text{Crit}(f_s)$ having the property that $\lambda(p)=0.$ As a consequence the above Theorem \hyperref[main main theorem]{\ref*{main main theorem}} together with Theorem \hyperref[minor theorem]{\ref*{minor theorem}} give the top homology groups of the Milnor fibre.\\

\subsection{Notation}
Before we prove the theorem we need some notation. We have the differential
$$\bar{d}(\omega\otimes \delta)=d\omega\otimes\delta+(-1)^p \omega\wedge \nabla\delta$$
where $d(\sum_{I,J} a(x,s, t) dx^I\wedge ds^J)=\sum_{I, J} \frac{\partial a(x,s,t)}{\partial x} dx^I\wedge dx\wedge ds^J+\sum_{I,J} \frac{\partial a(x,s,t)}{\partial s} dx^I\wedge ds^J \wedge ds$ is the usual relative differential. We have also the standard differential which we shall also denote by $d$: it acts on $\sum_{I, J} a(x,s,t) dx^I\wedge ds^J$ in the usual way and gives a form expressed in $dx, ds$ and $dt$. 
\subsection{The Proof}
\begin{proof}
\begin{itemize} 
\item If $i_s^{\ast}(\omega)$ is homotopic to $i_{0}^{\ast}(\omega)$ then
\begin{equation}
i_0^{\ast}(\omega)-i_s^{\ast}(\omega)=dh(\omega)+h(d\omega)
\end{equation}
which yields
\begin{equation}\label{eq}
\frac{i_0^{\ast}(\omega)}{t-f}-\frac{i_s^{\ast}(\omega)}{t-f_s}=\frac{dh(\omega)+h(d \omega)}{t-f}+ \frac{\langle s, x^2 \rangle i_s^{\ast}(\omega)}{(t-f)(t-f_s)}
\end{equation}
where $h(\omega)=\int_0^1 i_S(\omega)$ for $S$ the standard vector field $\partial/\partial s$ on $\mathbb{R}\cap \{0\leq s\leq s_0\}$. We shall first prove that (\hyperref[eq]{\ref*{eq}}) is zero in homology by proving something stronger namely that if $\omega\otimes\delta_{t-F}\in DR^{n-1}_{\mathcal{N}^{+}\times\mathbb{R}^{+}, \mathbb{R}^{+}}(\mathfrak{B}_F)$ is closed then $i_s^{\ast}(\omega)=0$ for all $s\in\mathbb{R}\cap\{0\leq s\leq s_0\}$.

\begin{enumerate}
\item[A.]\label{sit}
Let $\omega=a(x,s,t) dx_I\wedge ds_J$ be a relative $(n-1)$-form with $I\subset \{1,\dots, n-1\}$ and $J\subseteq\{j\}$. Assume that $i_s^{\ast}(\omega)$ is nonzero; then $J=\emptyset$ and as a consequence we can take $I=\{1,\dots, n-1\}$. Assume that $\omega\otimes \delta_{t-F}$ is closed; that is, assume that
\begin{equation}
\frac{d\omega}{t-F}+(-1)^{n-1}\frac{\omega\wedge dF}{(t-F)^2}=0.
\end{equation}
Then $i_s^{\ast}(\omega\otimes \delta_{t-F})$ is closed as well. Since
$$i_s^{\ast}(d\omega)=a_{x_n} dx_1\wedge\dots\wedge dx_n$$
$$i_s^{\ast}(\omega\wedge dF)=a(f_{x_n}+ 2s x_n )dx_1\wedge\dots \wedge dx_n$$
one obtains the differential equation
\begin{equation}\label{eq 2}
a_{x_n}(t-F)=(-1)^{n} a(f_{x_n}+2s x_n)
\end{equation}
However in equation (\hyperref[eq 2]{\ref*{eq 2}}) the degrees in the variable $t$ are $\deg_t(LHS)=\deg_t(a)+1$ whereas $\deg_t(RHS)=\deg_t(a)$ which is impossible. Therefore $i_s^{\ast}(\omega\wedge dF)=0$ and $i_s^{\ast}(d\omega)=0.$
Suppose that $a(x,s,t)\neq 0$ is nonzero. Then (\hyperref[eq 2]{\ref*{eq 2}}) implies that $a$ and $f$ respectively are solutions of the differential equations  
\begin{equation}\label{eq 3}
a_{x_n}=0,\qquad f_{x_n}+2sx_{n}=0
\end{equation}
for all $s\in \mathbb{R}\cap\{0\leq s\leq s_0\}$. But then if $\theta_s=s\sum_{i=1}^n x_i dx_i$ the second equation in (\hyperref[eq 3]{\ref{eq 3}}) means that $df(0,\dots, 0, 1)$ is parallel to $\theta_s(0,\dots, 0, 1)$ for all $s$ which is impossible since $df$ is independent of $s$. Therefore $a\equiv 0$ identically hence $\omega=0.$\\
\item[B.] Consider the complex $\Omega_{\mathcal{N}}(D_t)$ (see \cite{Pham}) with differential
$$\tilde{d}(\omega)=d\omega-\omega\wedge dF$$
where
$$d(\sum_{I, J, L} a(x,s) D_t^L dx^I ds^J)=\sum_{I, J} da(x, s) dx^I ds^J.$$
Then there is an isomorphism (see \cite{Pham}) of complexes\\ $DR^{\bullet}_{\mathcal{N}\times \mathbb{R}+, \mathbb{R}^+}(\mathfrak{B}_{F})\to \Omega^{\bullet}_{\mathcal{N}}(D_t)$ defined by $\delta_{t-F}\mapsto 1$.
\item[C. (another proof of $\omega=0$)]\label{sitt}
Using this isomorphism of complexes it follows that $i_s^{\ast}(d\omega)-i_s^{\ast}(\omega\wedge dF)=0$. This means that $a$ and $f$ satisfy the differential equation
\begin{equation}
a_{x_n}-a(f_{x_n}+2sx_n)=0
\end{equation}
hence if $a(x, s ,t) \neq 0$ then
\begin{equation}\label{final}
\frac{a_{x_n}}{a}=f_{x_n}+2sx_n.
\end{equation}
\begin{enumerate}
\item Suppose that $\deg_{x_n}(a_{x_n}/a)=0$. Then 
$$ \frac{a_{x_n}}{a}=c(x_I, s, t)$$
$$\implies a=k(x_I, s, t) e^{c(x_I, s, t) x_n}$$
which inserted into a solution of the differential equation (\hyperref[final]{\ref*{final}}) gives 
\begin{equation}\label{fin} f(x)=-sx_n^2+\ln k(x_I, s,t)+c(x_I, s,t) x_n.
\end{equation}
In (\hyperref[fin]{\ref*{fin}}) the $LHS$ is independent of $s$ so the same must be true of the $RHS$. Therefore the term of $\deg_{x_n}=2$ must cancel out in the $RHS$. But this is impossible because the other terms are of strictly smaller degree, contradiction.
\item Otherwise $\deg_{x_n}(a_{x_n}/a)=-1$ which implies $\deg(f_{x_n})=-1$
which is impossible since $f$ is analytic in the origin. Therefore $a(x,s,t) \equiv 0$ hence $\omega=0.$
\end{enumerate}
\end{enumerate}
\end{itemize}
\end{proof}

\begin{remark}
By definition of the pushforward of a differential form,
$$\pi_{\ast} i_s^{\ast}(\omega\otimes\delta_{t-F})=\int_{\pi^{-1}(t)} \psi \wedge dt$$
for a form $\psi(x)\in \Omega^{n-1}_{\mathcal{N}_s^{+}\times\mathbb{R}^{+},\mathbb{R}^{+}}$ such that $i_s^{\ast}(\omega)=\psi \wedge \pi^{\ast} dt$. Recall that locally on open sets one proceeds as follows. Since $\pi: (x, t)\to t$ is a submersion on  $\mathcal{N}_s\times\mathbb{R}$ if 
$$i_s^{\ast}(\omega\otimes\delta_{t-F})=\frac{a_s(x,t) dx_I}{t-f_s}$$
then $\pi^{\ast} dt=\pi dt$ and if $\Gamma_s: (x,t)\mapsto t-f_s(x)$ then by definition
$\psi(x):=\frac{a_s}{\Gamma_s}\circ \pi^{-1}(x) dx_I.$
Since the support of $\delta_{t-f_s}$ is contained in the graph $\Gamma_{f_s}$ of $f_s$ and since $f_s$ is everywhere nonzero on $\mathcal{N}_s$ this differential form is
$$\psi(x)=\frac{a_s(x, t)}{t-f_s} dx_I=i_s^{\ast}(\omega\otimes\delta_{t-F}).$$
identically. Returning to the pushforward one can as a consequence write
$$\pi_{\ast} i_s^{\ast}(\omega\otimes \delta_{t-F})=\int_{\pi^{-1}(t)} i_s^{\ast}\frac{ \omega}{t-F} \wedge dt.$$
\end{remark}

\section{An application}
\subsection{Public-Key Encryption}
Public key encryption schemes such as RSA and Diffie-Hellman has a long usage history. The main idea in such schemes is that the key used for encryption can be widely known without further ado, whereas the key used for decryption is secret and not to be leaked. A formal definition is in order:
\begin{definition}
A public key encryption scheme is a triple $(Gen, Enc, Dec)$ of probabilistic polynomial-time algorithms such that:\\
\begin{enumerate}
    \item $Gen: 1^n\mapsto (pk, sk)$ inputs the security parameter and outputs a public key $pk$ and a secret key $sk.$
    \item $Enc: (pk, \mathfrak{m})\mapsto c$ is probabilistic and inputs the public key and a message and outputs a ciphertext $c.$
    \item $Dec: (sk, c)\mapsto m\vee error$ is a deterministic algorithm which inputs the secret key and a ciphertext and outputs a message or an error message.
\end{enumerate}
It is furthermore required that $Dec_{sk}(Enc_{pk}(\mathfrak{m}))=\mathfrak{m}$ for all pairs $(pk, sk)$ except possibly with negligeable probability.
\end{definition}

Of course the security of such an encryption scheme hinges on the knowledge of the secret key $sk$. To give, however, a precise meaning to what is meant by security we shall restrict our attention to giving a brief discussion of $CCA$-security, which together with $CPA$-security forms the two perhaps most widely used notions of security for encryption schemes.
\subsection{CCA-security}
Security of a scheme $\Pi=(Gen, Enc, Dec)$ against \textit{Chosen-Ciphertext Attacks} or $CCA-security$ means that an attacker $\mathcal{A}$ is given the public key and access to encryption of any message. Security fails if the attacker then is able to obtain the secret key. In detail:

\begin{definition}
$PubK_{\Pi, \mathcal{A}}^{cca}(n)$ is the experiment:
\begin{enumerate}
    \item $Gen(1^n)$ is run to produce $(sk, pk).$
    \item The attacker algorithm $\mathcal{A}$ is given $pk$ and access to a decryption oracle $Dec_{sk}$. It outputs two messages $\mathfrak{m}_1, \mathfrak{m}_2\in \mathfrak{M}$ of equal length.
    \item A bit $b\in\{0, 1\}$ is chosen uniformly at random and the ciphertext $Enc_{pk}(m_b)\mapsto c$ is given to the attacker $\mathcal{A}.$
    \item The attacker $\mathcal{A}$ continues to interact with the decryption oracle $Dec_{sk}$ but cannot decrypt $c$. It finally outputs a bit $b'\in\{0,1\}$. If $b=b'$ the attacker succeeds and the experiment outputs $1$, otherwise it fails and it outputs $0$.
\end{enumerate}

One says that $\Pi$ is $CCA$-secure if for any probabilistic polynomial-time attacker algorithm $\mathcal{A}$ there exists a negligeable function $\epsilon(n)$ such that

$$Prob(PubK_{\Pi, \mathcal{A}}^{cca}(n)=1)\leq \frac{1}{2}+\epsilon(n).$$
\end{definition}

We are now ready to construct an encryption scheme based on morsification of real singularities.

\subsection{The Construction}
\subsubsection{Key-Generation} The scheme $\Pi=(Gen, Enc, Dec)$ has key generating algorithm
$$Gen: 1^n\mapsto (pk, sk)$$
$$pk=s, sk=\vec{\lambda}$$
Here $s\in \mathbb{R}$ is chosen uniformly at random in such a way that $|s|\leq s_0$ where $s_0\in\mathbb{R}$ is as in Theorem \hyperref[main main theorem]{\ref*{main main theorem}}.

\subsubsection{The message space} 
Let $\mathfrak{S}$ be a class of singularities parametrised by $\mathfrak{P}$.
\begin{hypothesis}\label{hyp}
    We assume that there exists a unique $k\in\mathbb{N}$ such that $k=H_{n-1}(\mathcal{F}_(f)$ where $f$ goes through all the function germs in $\mathfrak{S}.$
\end{hypothesis}

The message space is the set
$$\mathfrak{M} = \{\mathfrak{m}\in \mathfrak{P}| \exists ! f\in \mathfrak{S} \mathfrak{m}=\mathfrak{m}_f\}$$

In what follows we will take $\mathfrak{S}$ to be an appropriate subset of the set of quasi-homogeneous polynomials parametrised by their weight vectors.
Consider a message $\mathfrak{m}$ and let $f: (\mathbb{R}^{n+1}, 0)\to (\mathbb{R},0)$ be the corresponding quasi-homogeneous polynomial map germ. Let $f_s=f+Q_s(x)$
be a chosen morsification such that $f_s: \mathbb{R}^{n+1}\to \mathbb{R}$ is Morse and suppose furthermore that the morsification is chosen such that the conditions of the Theorem \hyperref[main main theorem]{\ref*{main main theorem}} are fulfilled. 
\subsubsection{The encryption algorithm}
$$Enc_{s}: \mathfrak{m}\mapsto \mathbf{c}$$
where $\mathbf{c}=(c_1,\dots c_l),$ for $0\leq l\leq n+1$, is the vector of critical points having Morse indices zero of the perturbation of the message $\mathfrak{m}$ given by the morsification $f_s$. The set of ciphertexts, denoted $\mathfrak{C},$ is therefore a subset of $\mathbb{R}^{n+1}.$
$Enc$ is probabilistic since the choice of $pk$ is.
\subsubsection{The decryption algorithm}:
$$Dec: \mathbf{c}\mapsto \phi(|\vec{\lambda}_s|-|\mathbf{c}| )$$
where $\phi: \mathbb{N}\to\mathfrak{S}$ is the function given by the hypothesis \hyperref[hyp]{\ref*{hyp}}. In detail
$$\phi: \text{rank} H_{n}(\mathcal{F}_{\tilde{f}}) \mapsto \tilde{f}$$
It is partially defined \footnote{note also that computationally speaking we only allow for bounded subsets of $\mathbb{N}$} but injective where it is defined so by the pidgeonhole principle it gives back $f$ and thus $\mathfrak{m}_f$.\\

By the Theorem \hyperref[main main theorem]{\ref*{main main theorem}} if $Enc_{pk}: \mathfrak{m}_f\mapsto \mathbf{c}$ where (per assumption) $|\mathbf{c}|=\vec{\lambda}_{0, s}$ then 
$$Dec_{sk}: \mathbf{c}\mapsto \phi(|\vec{\lambda}_{s}|-|\vec{\lambda}_{0,s}|)=\mathfrak{m}_f.$$
\begin{remark}
$Dec$ is clearly deterministic since the Theorem \hyperref[main main theorem]{\ref*{main main theorem}} is constructive.
\end{remark}
Clearly, the triple $\Pi$ thus constructed satisfies the conditions for being an encryption scheme.

\begin{theorem}
    The encryption scheme $(Gen, Enc, Dec)$ is $CCA$- secure.
\end{theorem}

\begin{proof}[Proof of $CCA$-security]\label{CCA}
The attacker $\mathcal{A}$ is given $s\in\mathbb{R}$ and access to oracle decryptions of ciphertexts in $\mathfrak{C}.$ But it is clear that regardless of how many decryptions $\mathbf{c}\in \mathbb{R}^k\mapsto \mathfrak{m}\in \mathbb{R}^{n+1}$ it is given the attacker is only able to deduce the number of Morse indices which are zero, which is insufficient for it to be able to deduce the secret key $\vec{\lambda}.$ 
\end{proof}

\begin{example} Take $f=\sum_{i=1}^{n+1} x_i^2$ so that $\mathfrak{m}=(1/2,\dots, 1/2)$. There is exactly one critical point and it has Morse index $n$ so $\vec{\lambda}=(n)$ is the secret key, $s=0$ is the public key and the ciphertext is the origin $(0,\dots, 0)\in \mathbb{R}^{n+1}$
\end{example}

This example and Shannon's theorem shows that for $n>0$ the above construction has not \emph{perfect secrecy}, since the key space is shorter than the message space. 

\begin{example}
Take $f=x^4-y^2$. Then $f(a^{1/4}x, a^{1/2}y)=ax^4-ay^2=a f(x,y)$ so that $\mathfrak{m}=(1/4, 1/2)$. Take 
$f_s=x^4-y^2+2s x^2$ with $s$ a positive real number. Then we get one critical point in the origin with index one so $\mathbf{c}=\emptyset$. The secret key is $\vec{\lambda}=(1)$.
\end{example}
\subsection{The Second Construction}
\subsubsection{The message space}
Let $\mathfrak{S}$ be a class of singularities. Assume that the hypothesis \hyperref[hyp]{\ref*{hyp}} holds and let
the message space be the set $\mathfrak{M}$ of the previous section.
\subsubsection{The ciphertext space}
Let $\vec{\lambda_s}$ denote the vector of Morse indices of a given $f$ under a morsification $\mathbb{R}^n\times \mathbb{R}\to\mathbb{R}$ $(x, s)\mapsto f_s(x).$ Let $\vec{\lambda_{0, s}}$ denote the vector of Morse indices of index zero. By the Theorem \hyperref[main main theorem]{\ref*{main main theorem}}
$k=|\vec{\lambda_s}|-|\vec{\lambda_{0, s}}|.$ The ciphertext space is the space $\mathfrak{C}$ of Morse indices of index $0.$
\subsubsection{The encyption scheme} 
The key space is $\mathcal{K}=\{pk, sk\}$ where $pk=s$ and $sk= \vec{\lambda_s}$. The key generation, encryption and decryption algorithms are
$$Gen: 1^{n}\mapsto s$$
$$Enc: (\mathfrak{m}_f, s)\mapsto \vec{\lambda_{0, s}}$$
and
$$Dec: (\vec{\lambda_{0}}, \vec{\lambda_s})\mapsto \phi(|\vec{\lambda_s}|-|\vec{\lambda_{0,s}}| )$$

Then $(Gen, Enc, Dec)$ is an asymmetric encryption scheme.

\begin{example}
  \[ x^2+sx\mapsto    
     \begin{cases}
       \emptyset\quad \text{if $s>0$}\\
       2\quad\text{if $s<0$}
     \end{cases} \] 
     \end{example}

\subsubsection{$CCA$-security}
The above encryption scheme is $CCA$-secure. The proof is the same as the proof of \hyperref[CCA]{\ref*{CCA}}.

\bibliographystyle{plain}
\bibliography{Gaussmanin.bib}

\begin{thebibliography}{1}

\bibitem{Lars}
Lars Andersen.
\newblock On {I}solated {R}eal {S}ingularities {I}.

\bibitem{barlet}
{D}aniel Barlet.
\newblock Isolated real singularities and asymptotic expansions for oscillating
  integrals.
\newblock 9(29 - 50), 2004.

\bibitem{Bjork}
Jan {E}rik {B}jörk.
\newblock {\em Analytic {D}-modules and {A}pplications}, volume~1 of {\em
  Mathematics and {I}ts {A}pplications}.
\newblock Springer {D}ordrecht.

\bibitem{italian}
{P}atrizia~{M}acr{\`i} Francesco~{Guaraldo} and {A}llessandro {T}ancredi.
\newblock {\em Topics on {R}eal {A}nalytic {S}paces}.
\newblock Advanced {L}ectures in {M}athematics. Vieweg \& {T}eubner {V}erlag,
  1986.

\bibitem{Kosinski}
Antoni~A. Kosinski.
\newblock {\em Differential {M}anifolds}.
\newblock Academic Press, Inc. Boston, MA, 1993.

\bibitem{Man58}
Yu.{I}. {M}anin.
\newblock Algebraic curves over fields with differentiation.
\newblock {\em Izv. {A}kad. {N}auk {S}{S}{S}{R} {S}er. {M}at.}, 22(6):737--756,
  1958.

\bibitem{Milnor}
John Milnor.
\newblock {\em Singular {P}oints of {C}omplex {H}ypersurfaces.}
\newblock Princeton University Press, Princeton, N.J, 1968.

\bibitem{Pham}
{F}rédéric Pham.
\newblock {\em {S}ingularités des {S}ystèmes {D}ifferentiels de
  {G}auss-Manin}, volume~2 of {\em Progress in {M}athematics}.
\newblock Springer, 1976.

\bibitem{Kulikov}
{V}alentine {S}.~Kulikov.
\newblock Mixed {H}odge {S}tructures and {S}ingularities.
\newblock 132, 1998.

\end{thebibliography}

\end{document}